\documentclass[a4paper]{amsart}
\usepackage{palatino,bbm,amssymb,amsmath,mathrsfs,amsthm}

\usepackage[v2,all]{xy}

\newcommand{\drpullback}[1][dr]{\save*!/#1-1.2pc/#1:(-1,1)@^{|-}\restore}

\newcommand{\Cc}{\mathscr{C}}
\newcommand{\Dd}{\mathscr{D}}

\newcommand{\Bb}{\mathscr{B}}
\newcommand{\Mm}{\mathscr{M}}
\newcommand{\Aa}{\mathscr{A}}
\newcommand{\Cat}{\mathscr{C}\hspace*{-0.04cm}at}
\newcommand{\Top}{\mathscr{T}\hspace*{-0.03cm}op}

\newcommand{\Set}{\mathscr{S}\hspace*{-0.06cm}et}
\newcommand{\Ab}{\mathscr{A}\hspace*{-0.04cm}b}

\newcommand{\Fun}{\mathscr{F}\hspace*{-0.06cm}un}
\newcommand{\Nat}{\mathscr{N}\hspace*{-0.09cm}at}

\newcommand{\Fib}{\mathscr{F}\hspace*{-0.04cm}ib}
\newcommand{\DiscFib}{\mathscr{D}\hspace*{-0.03cm}isc\Fib}
\newcommand{\PsdFun}{{\mathfrak{PsdFun}}}
\newcommand{\Shv}{{\mathfrak{Shv}}}
\newcommand{\Mod}{{\mathfrak{Mod}}}
\newcommand{\Rep}{{\mathfrak{Rep}}}
\newcommand{\dd}{\delta}
\newcommand{\ee}{\eta}

\newtheorem{theorem}{Theorem}[section]
\newtheorem{corollary}[theorem]{Corollary}
\newtheorem{proposition}[theorem]{Proposition}
\newtheorem{lemma}[theorem]{Lemma}
\theoremstyle{definition}
\newtheorem{definition}[theorem]{Definition}
\newtheorem{remark}[theorem]{Remark}
\newtheorem{example}[theorem]{Example}

\DeclareMathAlphabet{\mathbbe}{U}{bbold}{m}{n}
\newcommand{\standardsimplex}{\mathbbe{\Delta}}

\newcommand{\nn}{{\mathbb N}}
\newcommand{\Z}{{\mathbb Z}}

\newcommand{\F}{{\mathbb F}}
\newcommand{\R}{{\mathbb R}}
\newcommand{\A}{{\mathbb A}}

\newcommand{\op}{{op}}
\newcommand{\Hom}{\textup{Hom}}
\newcommand{\llim}{\textup{lim}}
\newcommand{\clim}{\mathop{\textup{colim}}\nolimits}
\newcommand{\Ext}{\textup{Ext}}
\newcommand{\Tor}{\textup{Tor}}

\newcommand{\h}{{\mathcal H}}
\newcommand{\Ner}[1]{{\mathcal N}(#1)}

\newcommand{\thomnat}{\mathfrak{NatS}^{GZ}}
\newcommand{\cochain}{\mathfrak{coChn}}
\newcommand{\thomcontra}{\mathfrak{NatS}_{GZ}}
\newcommand{\chain}{\mathfrak{Chn}}
\newcommand{\Qcoh}{\mathfrak{Qcoh}}

\newcommand{\symbhom}{\overline\Hom}
\newcommand{\symbten}{\underline\otimes}

\begin{document}

\title{Gabriel-Zisman cohomology and spectral sequences}
\thanks{The first author was partially supported by Spanish Ministry of Science and Catalan government grants
  PID2019-103849GB-I00, 
  2017 SGR 932, 
  MTM2017-90897-REDT, 
  MTM2016-76453-C2-2-P (AEI/FEDER, UE), 
  MTM2015-69135-P, 
and the third author by
MTM2016-76453-C2-2-P (AEI/FEDER, UE)
all of which are gratefully acknowledged. The second author thanks the Centre de Recerca Matem\`atica (CRM) in Bellaterra, Spain for inviting him during the research programme Homotopy Theory and Higher Categories (HOCAT), where this work was initiated.}

\author{Imma G\'alvez-Carrillo}
\address{Departament de Matem\`atiques,
  Universitat Polit\`ecnica de Catalunya}
\email{m.immaculada.galvez@upc.edu}
\author{Frank Neumann}
\address{School of Mathematics and Actuarial Science, Pure Mathematics Group, University of Leicester,
  University Road, Leicester LE1 7RH, United Kingdom}
  \email{fn8@le.ac.uk}
\author{Andrew Tonks}
\address{School of Mathematics and Actuarial Science, Pure Mathematics Group, University of Leicester,
  University Road, Leicester LE1 7RH, United Kingdom}
\email{apt12@le.ac.uk}

\date{}

\maketitle

\begin{abstract}
Extending constructions by Gabriel and Zisman, we develop a functorial framework for the cohomology and homology of simplicial sets with very general coefficient systems given by functors on simplex categories into abelian categories. Furthermore we construct Leray type spectral sequences for any map of simplicial sets. We also show that these constructions generalise and unify the various existing versions of cohomology and homology of small categories and as a bonus provide new insight into their functoriality.
\end{abstract}

\section*{Introduction}

\noindent The purpose of this article is to investigate systematically the functoriality of Gabriel-Zisman cohomology and homology of simplicial sets. Gabriel-Zisman (co)homology was introduced by the authors in~\cite{GNT2} inspired by constructions originally due to Thomason~\cite{Wei2}, Gabriel-Zisman~\cite{GZ} and Dress~\cite{Dr} in order to give a simplicial interpretation of the various (co)homology theories for small categories including Baues-Wirsching and Hochschild-Mitchell (co)homology (compare~\cite{BW,GNT,Mi}). Gabriel-Zisman (co)homology is defined for any simplicial set $X$ with most general coefficient systems given by functors from the associated simplex category $\Delta/X$ to a given abelian category $\Aa$.  More precisely, we will work here with general coefficient system functors from $\Delta/X$ with values in arbitrary abelian categories $\Aa$, which are complete with exact products when considering cohomology and which are cocomplete with exact coproducts when considering homology. In particular, all constructions will work just fine when using coefficient systems functors with values in the category $\Ab$ of abelian groups. It turns out that these general coefficient systems, which we call Gabriel-Zisman natural systems, provide a systematic framework to study the (co)homology of simplicial sets, especially with respect to general naturality and functoriality properties. In particular we will also show in a direct way how Thomason (co)homology of small categories can be interpreted as Gabriel-Zisman (co)homology using the nerve construction and how its functoriality and naturality properties are just direct consequences of those of Gabriel-Zisman (co)homology. Another advantage of our approach is that using duality we get at once both cohomology and homology theories for small categories and simplicial sets. Furthermore, we will construct Leray type spectral sequences for Gabriel-Zisman cohomology and homology for any map $f\colon X\rightarrow Y$ of simplicial sets and identify the lower terms of these spectral sequences for particular coefficient systems. The Leray-Serre spectral sequences in cohomology and homology for Kan fibrations of simplicial sets are specialisations of these general Leray type spectral sequences (compare~\cite{Dr,GZ}). We aim to use these general Leray type spectral sequences in the future for calculations in various different situations and frameworks from algebraic geometry, algebraic topology and category theory.

In a related homological context, Fimmel \cite{Fi} developed a theory of Verdier duality for a particular class of cohomological coefficient systems on simplicial sets, corresponding via geometric realisation to sheaves on topological spaces. Such a duality theory was first conjectured by Beilinson and allows for interesting applications for example to Beilinson's theory of local adeles~\cite{Be} and to buildings for representations of reductive algebraic groups over finite fields. We expect that our general functorial formalism and the construction of Leray type spectral sequences developed here will give new insights and calculational tools in these algebraic situations.

Similar constructions as those considered here could also be made for cubical instead of simplicial sets as indicated by recent work of Husainov on the homology of cubical sets~\cite{Hu,Hu2}.

The article is structured as follows: In the first section we will recall fundamental constructions from the theory of simplicial sets and then introduce the general concepts of Gabriel-Zisman cohomology and homology of simplicial sets, study their functorial properties and show how these constructions unify and generalise existing notions of cohomology and homology of small categories. We also discuss several interesting examples for future exploration and applications. In the second section we will construct Leray type spectral sequences in Gabriel-Zisman cohomology and homology for any map of simplicial sets within our general framework. We will then specialise the coefficient systems for particular situations to be able to identify the lower pages of these spectral sequences in more familiar terms. And finally, the classical Leray-Serre spectral sequences for cohomology and homology of a Kan fibration of simplicial sets will be derived as special cases.

\section{Gabriel-Zisman (co)homology of simplicial sets}

\subsection{Categories of simplices and simplex categories}
We will collect in this subsection several fundamental concepts from the theory of simplicial sets and small categories, which will be needed later (compare also the systematic accounts in~\cite{GZ,GJ,I,MacL,GM} and~\cite{Ri}).

Let $\Delta$ as usual be the category whose objects are the totally ordered finite sets $[m]=\{0<1<\cdots<m\}$
and whose morphisms are the order preserving functions $\theta\colon  [m]\rightarrow [n]$ between them. Alternatively, we can regard $\Delta$ as a full subcategory of the category  $\Cat$ of small categories, whose objects are the categories  $[m]=(0 \rightarrow 1\rightarrow \cdots \rightarrow m)$.

Among the morphisms  of $\Delta$ are the \emph{coface maps}
$$\dd^i\colon [n-1]\rightarrow [n], \;\;\; 0\leq i\leq n$$
$$\dd^i(0 \rightarrow 1\rightarrow \cdots \rightarrow n-1)=(0\rightarrow 1\rightarrow \cdots \rightarrow i-1\rightarrow i+1\rightarrow \cdots \rightarrow n),$$
composing the arrows $i-1\rightarrow i\rightarrow i+1$, and the \emph{codegeneracy maps}
$$\ee^j\colon  [n+1]\rightarrow [n], \;\;\; 0\leq j\leq n,$$
$$\ee^j(0 \rightarrow 1\rightarrow \cdots \rightarrow n+1)=(0\rightarrow 1\rightarrow \cdots \rightarrow j\rightarrow j\rightarrow \cdots \rightarrow n)$$
inserting the identity morphism $id_j$ in the $j$-th position. These morphisms $\dd^i$ and $\ee^j$ satisfy the usual cosimplicial identities and give a set of generators and relations for the category $\Delta$ (compare ~\cite{BK},~\cite{GJ} and~\cite{MacL}).

Let $\Cc$ be a category. A \emph{simplicial object} in $\Cc$ is a functor $X\colon  \Delta^{op}\rightarrow \Cc$. Dually, a \emph{cosimplicial object} in $\Cc$ is a functor $X\colon  \Delta\rightarrow \Cc$. In particular, if $\Cc=\Set$ is the category of sets a functor $X\colon \Delta^{op}\rightarrow \Set$ is called a \emph{simplicial set} and a functor $X\colon \Delta \rightarrow \Set$ a \emph{cosimplicial set}. Simplicial objects in a category $\Cc$ form a category $\Delta^{op}\Cc$, where the morphisms are natural transformations. Dually, we have the category of cosimplicial objects $\Delta\Cc$.

In the category $\Delta^{op}\Set$ of simplicial sets we can consider for every integer $n\geq 0$ the representable simplicial set $\Delta[n]=\Hom_{\Delta}(-, [n])$ called the \emph{standard $[n]$-simplex}.
The \emph{$n$-simplices} $X_n$ of a simplicial set $X$ are given as $X_n=X([n])$ and we sometimes will also write $X_{\bullet}=\{X_n\}_{n\geq 0}$ to denote a simplicial set. As usual, we will denote by $d_i\colon  X_n\rightarrow X_{n-1}$ for $0\leq i\leq n$ the \emph{face maps} and  by $s_j\colon  X_n\rightarrow X_{n+1}$ for $0\leq j\leq n$ the \emph{degeneracy maps}.

The Yoneda Lemma readily implies that the $n$-simplices of a simplicial set $X$ are in bijective correspondence with the morphisms of simplicial sets from $\Delta[n]$ to $X$ i.e., $X_n\cong \Hom_{\Delta^{op}\Set} (\Delta[n], X)$. Thus morphisms of simplicial sets $\Delta[n]\to\Delta[m]$ can be identified with morphisms $[n]\to[m]$ of $\Delta$ and vice versa.

\begin{definition} Let $X$ be a simplicial set. The \emph{category of simplices} or \emph{simplex category}
of $X$ is the comma category $\Delta/X$ whose objects are the simplices $x$ of $X$ and whose morphisms $x\to x'$ are morphisms $\theta$ of $\Delta$ such that $x=X(\theta)(x')$. Alternatively, the objects are pairs $([m],x)$, where $x\colon  \Delta[m]\rightarrow X$, and morphisms are commuting triangles,
$$
\xymatrix{\Delta[m] \ar[rr]^{\theta}\ar[rd]_{x} && \Delta[m'] \ar[ld]^{x'}\\
&X.}
$$
Given a map $f\colon X\to Y$ of simplicial sets there is a functor
$$
\Delta/f\colon \Delta/X\to \Delta/Y
$$
given by $(\Delta/f)(x)=fx$.
\end{definition}

The opposite category of the category of simplices $\Delta/X$ of $X$ can also be interpreted as the Grothendieck construction for the functor $X\colon  \Delta^{op}\rightarrow \Set$, that is, as the category
$$(\Delta/X)^\op=\int_{\Delta^{op}} X.$$
It also comes together with a natural projection functor
$$P^\op\colon  (\Delta/X)^\op=\int_{\Delta^{op}} X\rightarrow \Delta^\op,$$
which is a discrete Grothendieck fibration i.e.,\ a Grothendieck fibration where all the fibers are sets. In fact, every discrete Grothendieck fibration
$$P\colon  \Cc\rightarrow \Delta$$
can be obtained as the Grothendieck construction $\int_{\Delta} X$ of the functor
$$X\colon  \Delta^\op \rightarrow \Set,\,\, X([n])=X_n=P^{-1}([n]).$$
This gives an equivalence of categories
$$\DiscFib(\Delta)\stackrel{\simeq}\leftrightarrow \Delta^{op}\Set$$
between the category of {\em discrete}\/ Grothendieck fibrations over $\Delta$ and the category of simplicial sets, which is a very special case of the equivalence of $2$-categories between the $2$-category $\Fib(\Bb)$ of Grothendieck fibrations over a small category $\mathscr B$ and the $2$-category of contravariant pseudofuncors $\PsdFun(\Bb^{op}, \Cat)$ from $\mathscr B$ into the $2$-category $\Cat$ of small categories (see~\cite[2.3]{GNT2} and~\cite{SGA}).

Any contravariant functor from a small category $\Cc$ into the category of sets is a colimit of representable functors $\Hom_\Cc(-,c)$, and the Density Theorem \cite[Chap. III, \S 7, Thm. 1]{MacL} states that we can recover a simplicial set $X$ via the isomorphism
$$X\cong \clim\limits_{([m],x)\in\Delta/X}\Delta[m].$$

Let us here also recall that the \emph{nerve} $\Ner\Cc$ of a small category $\Cc$  is the simplicial set $\Ner\Cc_{\bullet}=\{\Ner\Cc_n\}$ whose $n$-simplices are given as
$$\Ner\Cc_n=\Hom_{\Cat}([n], \Cc).$$
In more concrete terms, an $n$-simplex $\sigma$ of $\Ner\Cc$ is just a string of  $n$ composable morphisms $\gamma_i$ in $\Cc$
$$\sigma=(C_0\stackrel{\gamma_1}\longrightarrow C_1\stackrel{\gamma_2}\longrightarrow \cdots \stackrel{\gamma_n}\longrightarrow C_n)$$
where $C_i=\sigma(i)$ are objects of $\Cc$.

The face and degeneracy maps $d_i$ and $s_j$ are then given by precomposition with the coface and codegeneracy maps $d^i$ and $s^j$. In other words, the value $d_i(\sigma)$ of the face map $d_i\colon  \Ner\Cc_n\rightarrow \Ner\Cc_{n-1}$ is obtained from $\sigma$ by omitting the object $C_i=\sigma(i)$, and by omitting $\gamma_1$ if $i=0$, composing $\gamma_{i+1}$ and $\gamma_i$ if $0<i<n$, or omitting $\gamma_n$ if $i=n$. Similarly, the value $s_j(\sigma)$ of the degeneracy map $s_j\colon  \Ner\Cc_n\rightarrow \Ner\Cc_{n+1}$ is obtained from $\sigma$ by repeating the object $C_j$ and inserting an identity morphism $id_{C_j}$.

The nerve construction defines a functor ${\mathcal N}\colon  \Cat\rightarrow \Delta^{op}\Set$ from the category of small categories to the category of simplicial sets.

We will finally define another simplex category, the simplex category of a small category.

\begin{definition}
Let $\Cc$ be a small category. The \emph{simplex category $\Delta/\Cc$ of $\Cc$} is the comma category whose objects are pairs $([m], f)$, where $[m]$ is an object of $\Delta$ and $f\colon [m]\to \Cc$ is a functor, and whose morphisms $([m],f)\longrightarrow ([n],g)$ are morphisms $\theta\colon [m]\to[n]$ of $\Delta$ with $f=g\circ \theta$.
\end{definition}

Thus objects $([m],f)$ of $\Delta/\Cc$ are elements of the simplicial nerve $\Ner \Cc$ of $\Cc$. We will often omit the $[m]$ from the notation and regard objects as diagrams or strings
$$f=(C_0\stackrel{f_1}\longrightarrow C_1\stackrel{f_2}\longrightarrow \cdots \stackrel{f_m}\longrightarrow C_m).$$
The morphisms of $\Delta/\Cc$ are as usual generated by omitting or repeating objects $C_i$ in such diagrams.

The simplex category $\Delta/\Cc$ of a small category $\Cc$ is therefore just the simplex category $\Delta/\Ner \Cc$ of the nerve $\Ner \Cc$ of $\Cc$.
It was shown by Illusie~\cite[VI.3]{I} and Latch~\cite{L} that the functor $\Delta/-\colon  \Delta^{op}\Set\rightarrow \Cat$ is in fact a weak homotopy inverse to the nerve functor $\mathcal N\colon  \Cat\rightarrow \Delta^{op}\Set$ i.e.,\ for any simplicial set $X$ there is a weak equivalence of simplicial sets $$\mathcal N(\Delta/X)\stackrel{\sim}\rightarrow X.$$

Another incarnation of the simplex category of $\Cc$ is given by the Grothendieck construction of the contravariant diagram of discrete categories given by the simplicial nerve,
$$(\Delta/\Cc)^\op\;\;\cong\;\;\int_{\Delta^\op}\Ner\Cc\qquad\text{ where }\qquad\Ner \Cc\colon \Delta^\op\to \Set\to\Cat.$$

Let $F\colon  \Cc\rightarrow \Dd$ be a functor. If $D$ is an object of $\Dd$, then the \emph{fiber $\Cc_D=F^{-1}(D)$ of $F$ over} $D$ is the subcategory of $\Cc$ whose objects are the objects $C$ of $\Cc$ such that $F(C)=D$ and whose morphisms are the morphisms $f\colon  C\rightarrow C'$ in $\Cc$ such that $F(f)=id_D$. The \emph{left fiber $\Cc/D=F/D$ of $F$ over} $D$ is the category of all pairs $(C, u)$ with $C$ an object of $\Cc$ and $u\colon  F(C)\rightarrow D$ a morphism in $\Dd$ and where a morphism $(C, u)\rightarrow (C', u')$ is given as a morphism $v\colon  C\rightarrow C'$ in $\Cc$ such that $u=u'\circ F(v)$. Dually, we have the notion of a \emph{right fiber $D/\Cc=D/F$ of $F$ over} $D$.

If $\Cc$ is a small category, then the simplex category $\Delta/\Cc$ is also given as the left fiber over $\Cc$ of the embedding $\Delta\rightarrow \Cat$.

More generally, let $\Cc$ be any category and $c$ an object of $\Cc$. Given any cosimplicial object in $\Cc$, that is, a functor $F\colon \Delta\to \Cc$,
one can define the \emph{simplex category $\Delta/c$} as the comma category whose objects are pairs $([m], f)$, where $[m]$ is an object of $\Delta$ and $f\colon F([m])\to c$ is an arrow of $\Cc$, and whose morphisms $([m],f)\longrightarrow ([n],g)$ are morphisms $\theta\colon [m]\to[n]$ of $\Delta$ with $f=g\circ F(\theta)$.
The definitions of the simplex category above are for the obvious functors $F\colon \Delta\to\Delta^{op}\Set$ and  $F\colon \Delta\to\Cat$. Note that both of these are fully faithful functors.


\subsection{Gabriel-Zisman (co)homology of simplicial sets and its functorial properties}
In this subsection we will present a systematic account of the constructions and fundamental functorial properties of cohomology and homology of simplicial sets with general coefficient systems. The coefficient systems described here were first introduced by Gabriel and Zisman~\cite[App. II.4]{GZ} to analyse the homology of simplicial sets and were also discussed systematically by Dress~\cite{Dr}. Fimmel \cite{Fi} also used these coefficient systems to construct a Verdier duality theory for sheaves on simplicial sets. As a particular application of our general framework we will show how Thomason cohomology and homology of small categories as introduced and studied by the authors in~\cite{GNT2} fits into this picture.

\begin{definition}
Let $X$ be a simplicial set and $\Mm$ be a category. A functor $T\colon  \Delta/X\rightarrow \Mm$ is called a \emph{(covariant) Gabriel-Zisman  natural system on $X$ with values in} $\Mm$.
\end{definition}

\begin{remark}\label{sheafrem}
  A  Gabriel-Zisman natural system $T$ will be termed a \emph{sheaf}  if $T(x\stackrel\theta\to x')$ is an isomorphism in $\Mm$ whenever $\theta$ is a codegeneracy map $s^i\colon[n+1]\to[n]$ in $\Delta$ (or equivalently, whenever $\theta$ is surjective, cf.\ \cite[Definition 3.2]{Fi}). 
  \end{remark}

We now define a general cohomology theory for simplicial sets using these Gabriel-Zisman natural systems as coefficients.

\begin{definition}\label{drcochn}
Let $X$ be a simplicial set and let $T\colon  \Delta/X\rightarrow \Aa$ be a Gabriel-Zisman natural system with values in a complete abelian category $\Aa$ with exact products. The \emph{Gabriel-Zisman cochain complex} $C_{GZ}^*(X, T)$ of $X$ is defined as
$$C_{GZ}^n(X, T):=\prod_{\sigma_n\in X_n} T(\sigma_n),$$
for each integer $n\geq 0$, with differential
$$d=\sum_{i=0}^{n+1}(-1)^i d^i\colon
\prod_{\sigma_n\in X_n}T(\sigma_n) \longrightarrow
\prod_{\sigma_{n+1}\in X_{n+1}}T(\sigma_{n+1}).$$
The components of these $d^i$ are the morphisms
$$\delta^i_\#\colon  T(\sigma_{n+1}\circ \delta^i)\rightarrow T(\sigma_{n+1})$$
induced by the coface maps  $\delta^i\colon [n]\to[n+1]$.
The \emph{$n$-th Gabriel-Zisman  cohomology} of $X$ is the cohomology of this cochain complex,
$$H^n_{GZ}(X, T)=H^n(C_{GZ}^*(X, T)
, d).$$
\end{definition}

Equivalently, $C_{GZ}^*(X,T)$ is the cochain complex associated to the cosimplicial object
$$
\def\cccc#1{\displaystyle\!\!\prod_{\sigma_{#1}\in X_{#1}}\!\!\!T(\sigma_{#1})}
\xymatrix @C=40pt{
\cccc0
\ar @<-.5ex>[r]_-{d^1}\ar @<.5ex>[r]^-{d^0}
&
\cccc1
\ar @/_20pt/[l]_-{s^0}
\ar @<-1ex>[r]_-{d^2}\ar[r]|-{d^1}\ar @<1ex>[r]^-{d^0}
&
\cccc2
\ar @/_20pt/[l]|{s^1}\ar @<-1ex>@/_20pt/[l]_-{s^0}
\ar @<-1.5ex>[r]
\ar @<-.5ex>[r]
\ar @<.5ex>[r]
\ar @<1.5ex>[r]
&
\cccc3
\ar @/_20pt/[l]
\ar @<-1ex>@/_20pt/[l]
\ar @<-2ex>@/_20pt/[l]
}\dots
$$
given as the cosimplicial replacement $\prod^*T$ of the functor $T\colon \Delta/X\to \Aa$
(see~\cite[XI.5]{BK},~\cite{Wei2}).

For any simplicial set $X$ and a complete abelian category $\Aa$ with exact products, let $\thomnat$ be
the category whose objects are the (covariant) Gabriel-Zisman natural systems $T\colon \Delta/X\to\Aa$ with values in $\Aa$. A morphism $(\varphi,\tau)\colon T_X\to T_Y$ between Gabriel-Zisman natural systems $\Delta/X\stackrel{T_X}\longrightarrow\Aa$ consists of a morphism $\varphi\colon Y\to X$ of simplicial sets together with  a natural transformation $\tau\colon T_X\circ\Delta/\varphi\longrightarrow T_Y$. The composition of morphisms is given by the following diagram
$$
\xymatrix @R3em@C5em{\Delta/X\drto_{T_X}&\lto_{\Delta/\varphi} \dto|(0.45){\stackrel\tau\Longrightarrow\quad T_Y\quad\stackrel\upsilon\Longrightarrow}\Delta/Y&\lto_{\Delta/\psi} \Delta/Z\dlto^{T_Z}\\&\Aa\,.}
$$
The Gabriel-Zisman cochain complex defines in fact a functor
$$C_{GZ}^*\colon \thomnat\to\cochain, \quad C^*_{GZ}(T):=C^*_{GZ}(X, T),$$
from the category of Gabriel-Zisman natural systems with values in the abelian category $\Aa$, to the category of cochain complexes in $\Aa$.
The functor $C^*_{GZ}$ is defined on objects as above, and on morphisms by
\begin{align*}C^*_{GZ}(\varphi,\tau)\colon C^*_{GZ}(X,T_X)&\longrightarrow C^*_{GZ}(Y,T_Y),\\
(a_f)_{[n]\stackrel f\to X}
&\longmapsto
(\tau_g(a_{\varphi\circ g}))_{[n]\stackrel g\to Y}.
\end{align*}
Gabriel-Zisman cohomology therefore becomes a functor from $\thomnat$ to the category of graded objects in the category $\Aa$.
In fact, the correspondence $$T\longmapsto H^*_{GZ}(X, T)$$ is a cohomological $\partial$-functor on the category $\thomnat$
of (covariant) Gabriel-Zisman natural systems.\\

Dually, we define homology of simplicial sets with coefficients in {\em contravariant}\/ Gabriel-Zisman natural systems. These coefficients are in fact the original ones used by Gabriel and Zisman~\cite[App. III.4]{GZ} and Dress~\cite{Dr}.

\begin{definition}
Let $X$ be a simplicial set and $\Mm$ be a category. A functor $T\colon  (\Delta/X)^{op}\rightarrow \Mm$ is called a \emph{(contravariant) Gabriel-Zisman  natural system on $X$ with values in} $\Mm$.
\end{definition}

Using these general coefficient systems, we define now the Gabriel-Zisman homology of a simplicial set $X$.

\begin{definition}\label{drhn}
Let $X$ be a simplicial set and let $T\colon  (\Delta/X)^{op} \rightarrow \Aa$ be a contravariant Gabriel-Zisman natural system with values in a cocomplete abelian category $\Aa$ with exact coproducts.  The \emph{Gabriel-Zisman chain complex} $C^{GZ}_*(X, T)$ of $X$ is defined as
$$C^{GZ}_n(X, T):=\bigoplus_{\sigma_n\in X_n} T(\sigma_n),$$
for each integer $n\geq 0$, with differentials
\begin{eqnarray*}
d_n\colon
\bigoplus_{\sigma_{n+1}\in X_{n+1}}
  \!\!\!\!
T(\sigma_{n+1}) &\longrightarrow &
\bigoplus_{\sigma_n\in X_n} \!\! T(\sigma_n)\\
a_f&\mapsto&\sum_{i=0}^{n+1}(-1)^i(\delta^i)^\#
(a_{f}),
\end{eqnarray*}
where $(\delta^i)^\#\colon T(\sigma_{n+1})\to T(\sigma_{n+1}\circ\delta^i)$ is induced by the coface map $\delta^i\colon [n]\to[n+1]$.
The \emph{$n$-th Gabriel-Zisman homology} of $X$ is defined as the homology of this chain complex,
$$H_n^{GZ}(X, T):=H_n(C^{GZ}_*(X, T), d).$$
\end{definition}

Again, the Gabriel-Zisman chain complex is just the chain complex corresponding to a certain  simplicial object in $\Aa$, given by the simplicial replacement of $T$.

Let $\thomcontra$ be the category with objects the \emph{contravariant} Gabriel-Zisman natural systems $T\colon (\Delta/X)^\op\to\Aa$, and in which a morphism $(\varphi,\tau)\colon X\to Y$ is given by a functor $\varphi\colon X\to Y$ together with a natural transformation $\tau\colon T_X\to T_Y\circ\Delta/\varphi$.
The composition of morphisms is described by the following diagram,
$$
\xymatrix @R3em@C5em{(\Delta/X)^{op}\drto_{T_X}\rto^{\Delta/\varphi}& \dto|(0.45){\stackrel\tau\Longrightarrow\quad T_Y\quad\stackrel\upsilon\Longrightarrow}(\Delta/Y)^{op}\rto^{\Delta/\psi}& (\Delta/Z)^{op}\dlto^{T_Z}\\&\Aa\,.}
$$
The Gabriel-Zisman chain complex defines a functor
$$C^{GZ}_*\colon \thomcontra\to \chain, \quad C_*^{GZ}(T):=C_*^{GZ}(X, T),$$
where for morphisms we define
$$C^{GZ}_*(\varphi,\tau)\colon C^{GZ}_*(X, T_X)\longrightarrow C^{GZ}_*(Y, T_Y)$$
using the maps
$$
\tau_f\colon T_X(f)\longrightarrow T_Y(\varphi\circ f).
$$
Dually, Gabriel-Zisman homology therefore defines a functor from $\thomcontra$ to the category of graded objects in $\Aa$.\\

We will give now another interpretation of Gabriel-Zisman (co)homology, which is useful for analyzing its functorial properties.

Given a cosimplicial object in an abelian category $\Aa$ i.e.,\ a functor
$$F\colon \Delta\rightarrow \Aa$$
we have the associated cochain complex $(C^*(F), d)$ of $F$ defined as
$$C^n(F)\colon = F([n])$$
for each integer $n\geq 0$, with differential
$$d=\sum_{i=0}^{n+1}(-1)^i d^i\colon  F([n])\rightarrow F([n+1]),$$
where $d^i=F(\delta^i)$ and $\delta^i\colon [n]\rightarrow [n+1] $ for $0\leq i\leq n+1$ are the respective coface maps. We can now define the $n$-th cohomology of the cosimplicial object $F$ as the cohomology of the associated cochain complex
$$H^n(F):=H^n(C^*(F), d).$$
We therefore get a sequence of functors
$$H^*=(H^n)_{n\in \nn}\colon \Fun(\Delta, \Aa)\rightarrow \Aa, \,\, F\mapsto H^n(F).$$
Now we consider the following general situation. Let $\Cc$ be a small category and $\Aa$ a complete abelian category. Given two functors $P\colon  \Cc\rightarrow \Delta$ and $T\colon \Cc\rightarrow \Aa$, we have the right Kan extension $Ran_P(T)$ of the functor $T$ along $P$ (see~\cite[Chap. X]{MacL}):
$$
\xymatrix @R4em@C5em{\Cc\dto_{T}\rto^P&
\Delta\ar @{-->}[dl]^{Ran_P(T)}\\\Aa}
$$
It is an object of the functor category $\Fun(\Delta, \Aa)$, in other words a cosimplicial object of the abelian category $\Aa$ and we define:
\begin{definition}\label{defCoRan}  Let $\Cc$ be a small category and $\Aa$ a complete abelian category.
Given two functors $P\colon  \Cc\rightarrow \Delta$ and $T\colon \Cc\rightarrow \Aa$ {\em the $n$-th cohomology of $P$ with coefficients in $T$}\/ is defined as
$$H^n(P, T):= H^n(Ran_P(T)).$$
\end{definition}

Dually, given now a simplicial object of an abelian category $\Aa$ i.e.,\ a functor
$$F\colon \Delta^{op}\rightarrow \Aa$$
we have the associated chain complex $(C_*(F), d)$ of $F$ defined as
$$C_n(F):= F([n])$$
for each integer $n\geq 0$, with differential
$$d=\sum_{i=0}^{n+1}(-1)^i d_i\colon  F([n+1])\rightarrow F([n]),$$
where $d_i=F(\delta^i)$ and $\delta^i\colon [n]\rightarrow [n+1] $ for $0\leq i\leq n+1$ are the respective coface maps. So we can define the $n$-th homology of the simplicial object $F$ as the homology of the associated chain complex
$$H_n(F):=H_n(C_*(F), d).$$
We therefore get a sequence of functors
 $$H_*=(H_n)_{n\in \nn}\colon  \Fun(\Delta^{op}, \Aa)\rightarrow \Aa, \,\, F\mapsto H_n(F).$$
Now we consider the following general situation. Let $\Cc$ be a small category and $\Aa$ a cocomplete abelian category. Given two functors $P\colon  \Cc\rightarrow \Delta$ and $T\colon \Cc^{op}\rightarrow \Aa$, we have the left Kan extension $Lan^P(T)$ of the functor $T$ along $P^{op}$ (see~\cite[Chap. X]{MacL}):
$$
\xymatrix @R4em@C5em{\Cc^{op}\dto_{T}\rto^{P^{op}}&
\Delta^{op}\ar @{-->}[dl]^{Lan^P(T)}\\\Aa}
$$
It is an object of the functor category $\Fun(\Delta^{op}, \Aa)$, in other words a simplicial object of the abelian category $\Aa$ and we define:
\begin{definition}\label{defCoLan}  Let $\Cc$ be a small category and $\Aa$ a cocomplete abelian category.
Given two functors $P\colon  \Cc\rightarrow \Delta$ and $T\colon \Cc^{op}\rightarrow \Aa$ {\em the $n$-th homology of $P$ with coefficients in $T$}\/ is defined as
$$H_n(P, T):= H_n(Lan^P(T)).$$
\end{definition}

Now we would like to interpret this general cohomology and homology as a certain Ext and Tor construction, and in order to do so recall the following constructions
(compare~\cite[1.3, Remark 1.7]{GNT2} and~\cite[X.4]{MacL}, ):

\begin{definition}Let $\Ab$ be the category of abelian groups, and $\Aa$ an additive category. Then
\begin{enumerate}
\item The category $\Aa$ is \emph{cotensored} over  $\Ab$ if there is a functor $\Hom\colon \Ab^{op}\times\Aa\to \Aa$, satisfying the natural exponential law
$$
\Hom_\Aa(a, \Hom(A, b))\cong \Hom_{\Ab}(A, \Hom_\Aa(a,b)).
$$
\item The category $\Aa$ is \emph{tensored} over $\Ab$ if there is a functor $\otimes\colon \Ab\times\Aa\to \Aa$ satisfying the natural exponential law
$$
 \Hom_\Aa(A\otimes a, b)\cong \Hom_{\Ab}(A, \Hom_\Aa(a,b)).
$$
\end{enumerate}

Let $F\colon \Cc\to\Ab$, $T\colon \Cc\to\Aa$ be diagrams over $\Cc$.
The \emph{symbolic hom} $\symbhom_{\Cc}(F,T)$ as an object of $\Aa$ is determined by
natural isomorphisms
$$
\Hom_\Aa(a,\symbhom_{\Cc}(F,T))\cong \Nat(F,\Hom_\Aa(a,T(-))).
$$
Dually, for diagrams $F\colon \Cc\to\Ab$, $T\colon \Cc^{op}\to\Aa$, the \emph{symbolic tensor product} $F\symbten_{\Cc}T$ as an object of $\Aa$ is determined by
natural isomorphisms
$$
\Hom_\Aa(F\symbten_{\Cc}T, b)\cong \Nat(F,\Hom_\Aa(T(-),b)).
$$
\end{definition}

Now let
 $\underline{\Z}\colon \Cc\to \Ab$
be the constant diagram with value $\Z$. Then following the arguments and their duals in~\cite[1.3]{GNT2} we have for any diagram $T\colon \Cc\to \Aa$ that
$$\symbhom_{\Cc}(\underline{\Z},T)\cong \lim_{\Cc}T$$
and for any diagram  $T\colon \Cc^{op}\to \Aa$ we have dually
$$\underline{\Z}\symbten_{\Cc}T\cong \clim_{\Cc^{\op}}T.$$

Recall that a \emph{resolution} of $\underline\Z$ is a functor
$B_*\colon \Delta^\op\to \Fun(\Cc,\Ab)$ such that, for each object $c$ of $\Cc$, the reduced homology groups of the complexes $B_*(c)$ are trivial.
A resolution is \emph{free} if for each $n$ the functor $B_n\colon \Cc\to\Ab$ is a coproduct of representable functors $\Z\Hom(c,-)$.

We now express the general cohomology and homology constructions introduced above as derived functors of $\lim$ and $\clim$.
Suppose that $\Aa$ is a complete abelian category, with exact products. Let $B_*$ be a \emph{free resolution} of $\underline\Z$. Then the derived functors of  $\symbhom_{\Cc}(\underline{\Z},-)\cong \lim_{\Cc}(-)$ are given by the cohomology of the following Ext complex,
$$
\Ext_{\Cc}^*(\underline\Z,-):=\;\symbhom_{\Cc}(B_*,-).
$$
Dually, suppose that $\Aa$ is a cocomplete abelian category, with exact coproducts. Then the derived functors of $\underline{\Z}\symbten_{\Cc}(-)\cong\clim_{\Cc^{op}}(-)$  are given by the homology of the following Tor complex,
$$
\Tor_*^{\Cc}(\underline{\Z}, -):=B_*\symbten_{\Cc}(-).
$$

\begin{theorem}\label{resolution}
Let $\Aa$ be an additive category.
For any functor $P\colon \Cc\to\Delta$ there exists a resolution $B_*^P$
of the constant functor $\underline\Z$ such that
\begin{enumerate}
\item
If $\Aa$ is complete and cotensored over $\Ab$, then
$$Ran_P(T)\;\cong\;\symbhom_{\Cc}(B_*^P,T)
\;\colon \Delta\to\Aa$$
natural in $T\colon \Cc\to\Aa$.
\item
If $\Aa$ is cocomplete and tensored over $\Ab$, then
$$Lan^P(T)\;\cong\;B_*^P\symbten_{\Cc}T\;\colon \Delta\to\Aa$$
natural in $T\colon \Cc^{op}\to\Aa$.
\item
If $P\colon \Cc\to\Delta$ is a discrete fibration over $\Delta$, then there is a natural isomorphism
$$
B_n^P\cong\bigoplus_{c\in\Cc:P(c)=n}\Z\Hom_{\Cc}(c,d)
$$
and hence $B_*^P$ is a free resolution of $\underline\Z$.
\end{enumerate}
\end{theorem}

\begin{proof}
We set $B_*^P=\Z\Hom_\Delta(-, P(-))\colon \Delta^\op\to\Fun(\Cc,\Ab)$ and observe that the functor $B_*^P(d)=\Z\Hom_\Delta(-, P(d))\colon \Delta^\op\to\Ab$ is contractible since it is the standard simplex of dimension $P(d)$. We therefore have a resolution of $\underline\Z$. The natural isomorphisms of (1) and (2) now follow by expressing the Kan extensions and symbolic hom and tensor functors in terms of (co)ends:
\begin{align*}
Ran_P(T)&\cong
\int_{d\in \Cc}
\Hom(\Z\Hom_\Delta(-, P(d)),\,T(d))
=
\int_{d\in \Cc} \Hom(B_*^P(d),T(d))
\\&
\qquad\qquad\qquad\qquad\qquad\qquad\qquad\qquad\qquad\qquad
\cong \symbhom_{\Cc}(B_*^P,T),\\
Lan^P(T)&\cong
\int^{d\in \Cc^\op}\Z\Hom_{\Delta^\op}(P^\op(d),-)\otimes T(d)
=
\int^{d\in \Cc^\op}B_*^P(d)\otimes T(d)
\\&
\qquad\qquad\qquad\qquad\qquad\qquad\qquad\qquad\qquad\qquad
\cong B_*^P\symbten_{\Cc}T.
\end{align*}
If $P\colon \Cc\to\Delta$ is a discrete fibration over $\Delta$ there is a natural bijection
$$\Hom_\Delta(n,P(d))\cong \coprod_{c\in\Cc:P(c)=n} \Hom_\Cc(c,d)$$
for each $n\geq0$ and each object $d$ of $\Cc$. Thus
$$
B_n^P(d)\;=\;\Z\Hom_\Delta(n,P(d))
\;\cong\;\bigoplus_{c\in\Cc:P(c)=n}\Z\Hom_{\Cc}(c,d)
$$
and therefore the resolution $B_*^P\colon  \Delta^\op\to \Fun(\Cc,\Ab)$ of $\underline\Z$ is free.
\end{proof}
The following is then immediate:
\begin{corollary}
  Let $\Cc$ be a small category and  $\Aa$ an additive category, and let $P\colon \Cc\to\Delta$ be a discrete fibration.
  \begin{enumerate}
  \item
    If $\Aa$ is complete, with exact products, and $T\colon \Cc\to \Aa$ a functor,
    then the cohomology groups of $P$ with coefficients in $T$ are derived functors,
    $$H^n(P, T)=H^n(Ran_P(T))\cong \Ext_{\Cc}^n(\underline{\Z}, T)\cong\llim^n_{\Cc}T=H^n(\Cc, T).$$
  \item
       If $\Aa$ is cocomplete, with exact coproducts,  and $T\colon \Cc^\op\to \Aa$ a functor,
    then the homology groups of $P$ with coefficients in $T$ are derived functors,
    $$H_n(P, T)=H_n(Lan^P(T))\cong \Tor_{\Cc}^n(\underline{\Z}, T)\cong\clim_n^{\Cc}T=H_n(\Cc, T).$$
\end{enumerate}
\end{corollary}

As noted earlier, the discrete fibrations $P\colon \Cc\to\Delta$ are just given as the projections $P_X\colon \Delta/X\to \Delta$ from the simplex category of a simplicial set $X$.

\begin{theorem}\label{GZExt} Let $X$ be a simplicial set and let $T\colon  \Delta/X\rightarrow \Aa$ be a Gabriel-Zisman natural system with values in a complete abelian category $\Aa$ with exact products. The cohomology of $P_X\colon \Delta/X\to \Delta$ coincides with the Gabriel-Zisman cohomology of $X$,
 $$
C^*(P_X,T)\cong C_{GZ}^*(X,T),\qquad
H^*(P_X,T)\cong H_{GZ}^*(X,T),\qquad
$$
and Gabriel-Zisman cohomology may be identified as a derived functor,
$$H_{GZ}^n(X, T)\cong \Ext_{\Delta/X}^n(\underline{\Z}, T)\cong\llim^n_{\Delta/X}T=H^n(\Delta/X, T).$$
\end{theorem}

\begin{proof} From~\cite[Appendix II.4]{GZ}, it follows that the right Kan extension $Ran_{P_X}(T)$ of $T$ along the forgetful functor $P_X\colon \Delta/X\to \Delta$,
$$
\xymatrix @R4em@C5em{\Delta/X\dto_{T}\rto^{P_X}&
\Delta\ar @{-->}[dl]^{\,\,\, Ran_{P_X}(T)=\prod^*T}\\\Aa}
$$
is precisely the cosimplicial replacement $\prod^*T$ of the functor $T\colon \Delta/X\to \Aa$.

So we apply the above Theorem~\ref{resolution} to the right Kan extension $Ran_{P_X}(T)$ and use the identification of the cosimplicial replacement  $\prod^*T$ with the Gabriel-Zisman cochain complex $C_{GZ}^*(X,T)$ as constructed above to get the desired isomorphisms.

The last isomorphism is just the usual identification of the derived functors of the limit functor $\llim^n_{\Delta/X}T$ with the cohomology of the category $\Delta/X$ with coefficients in $T$ (see~\cite{Q,Ro} or~\cite{GNT}).
\end{proof}

Dually, we also have a similar isomorphism for Gabriel-Zisman homology of simplicial sets.

\begin{theorem}\label{GZTor} Let $X$ be a simplicial set and let $T\colon  (\Delta/X)^{op}\rightarrow \Aa$ be a Gabriel-Zisman natural system with values in a cocomplete abelian category $\Aa$ with exact coproducts. The homology of $P_X\colon  \Delta/X\rightarrow X$ coincides with the Gabriel-Zisman homology of $X$,
$$
C_*(P_X,T)\cong C^{GZ}_*(X,T),\qquad
H_*(P_X,T)\cong H^{GZ}_*(X,T),\qquad
$$
and Gabriel-Zisman homology may be identified as a derived functor,
$$H^{GZ}_n(X, T)\cong \Tor^{(\Delta/X)^{op}}_n(\underline{\Z}, T)\cong \clim_n^{(\Delta/X)^{op}}T=H_n((\Delta/X)^{op}, T).$$
\end{theorem}

\begin{proof} This is basically~\cite[Proposition 4.2]{GZ}. Alternatively, we can argue dually along the same lines as in the proof of Theorem~\ref{GZExt} using the resolution of the constant functor  $\underline{\Z}$ involving the dual notions, namely the symbolic tensor product functor and its derived Tor-functor for contravariant Gabriel-Zisman natural systems $T\colon  (\Delta/\Cc)^{op}\rightarrow \Aa$.
\end{proof}

Let us now look at several examples to illustrate the broad realm of applications and the necessity for the use of general Gabriel-Zisman natural systems as cohomological coefficient systems in contrast to more specialised coeffcients. The Leray type spectral sequences constructed in the following sections will then provide useful computational tools in all these frameworks of examples.

\begin{example}[Thomason (co)homology of categories]
We can interpret Thomason (co)homology of categories as introduced by the authors in~\cite{GNT2} both in terms of Gabriel-Zisman (co)homology of simplicial sets and via (co)homology of Kan extensions.

Let $\Cc$ be a (small) category, $\Aa$ be a complete abelian category with exact products and $T\colon  \Delta/\Cc\rightarrow \Aa$  a (covariant) Thomason natural system. From the general discussions above we see immediately that there are natural isomorphisms
$$H_{Th}^*(\Cc, T)\cong H^*(P_{\Cc}, T)\cong H^*_{GZ}(\Ner\Cc, T),$$
where $P=P_\Cc\colon  \Delta/\Cc\rightarrow \Delta$ is the forgetful functor and by identifying the categories of simplices $\Delta/\Cc=\Delta/\Ner\Cc$, where $\Ner\Cc$ is the nerve of the category $\Cc$. Here the notion of a Gabriel-Zisman natural system $T\colon \Delta/\Ner\Cc\to\Aa$ coincides with that of a Thomason natural system as we can readily identify the simplex category $\Delta/\Cc$ of $\Cc$ with the category of simplices $\Delta/\Ner\Cc$ over the simplicial nerve of $\Cc$ (see~\cite{GNT2}).

Dually, if $\Aa$ is a cocomplete abelian category with exact coproducts and given a (contravariant) Thomason natural system $T\colon  (\Delta/\Cc)^{op}\rightarrow \Aa$ we have natural isomorphisms
$$H^{Th}_*(\Cc, T)\cong H_*(P_{\Cc}, T)\cong H_*^{GZ}(\Ner\Cc, T).$$
As discussed in detail in~\cite{GNT} and~\cite{GNT2}, Thomason (co)homology generalises all the other (co)homology theories for small categories in the literature, including Baues-Wirsching and Hochschild-Mitchell (co)homology (compare for example~\cite{BW,CR,Mi,pr,Q}). Therefore the functoriality properties of these (co)homology theories are direct consequences of those of Gabriel-Zisman (co)homology as discussed above.
\end{example}

\begin{example}[Sheaves on topological spaces]\label{ExSheaves}
Let $X_{\bullet}$ be a simplicial set. We have the geometric realisation functor
$$|\hspace*{0.2cm}|\colon \Delta^{op}\Set\rightarrow \Top$$
given on objects as a coend or colimit as follows (see \cite{GZ,BK})
$$|X_{\bullet}|=\int^{[n]} X_n\times \standardsimplex^n=\clim \bigg (\coprod_{[n]\rightarrow [m]} X_m\times \standardsimplex^n \rightrightarrows \coprod_{[n]} X_n\times \standardsimplex^n\bigg )$$
where $\standardsimplex^n$ is the topological standard $n$-simplex in $\R^{n+1}$. We have $\standardsimplex^n=|\Delta[n]|$. It turns out that $|X_{\bullet}|$ is a compactly generated Hausdorff topological space. Let $R$ be a noetherian ring and $\Shv(|X_{\bullet}|)$ be the abelian category of sheaves of $R$-modules over $|X_{\bullet}|$. Let $\sigma\in X_n$ i.e., $\sigma\in \Hom_{\Delta^{op}\Set} (\Delta[n], X_{\bullet})$. We get an induced continuous map
$$|\sigma|\colon \standardsimplex^n\rightarrow |X_{\bullet}|.$$
Let $\mathcal{F}\in \Shv(|X_{\bullet}|)$ be a sheaf on $|X_{\bullet}|$ and assume that the inverse image sheaf $|\sigma|^*\mathcal{F}$ is constant on the subset $\mathrm{inn}(\standardsimplex^n)$ of inner points of the topological space $\standardsimplex^n$ for every simplex $\sigma\in X/\Delta$. Let $\mathcal{F}_{\sigma}$ denote the stalk of $\mathcal{F}$ at such an inner point. Then
$$F\colon \Delta/X_{\bullet}\rightarrow R-\Mod, \,\,\, F(\sigma)=\mathcal{F}_{\sigma}$$
is a Gabriel-Zisman (covariant) natural system, which in fact is a sheaf and Gabriel-Zisman cohomology $H^*_{GZ}(X_{\bullet}, F)$ gives sheaf cohomology $H^*(|X_{\bullet}|, \mathcal{F})$ of the topological space $|X_{\bullet}|$. In fact, when starting with a general Gabriel-Zisman (covariant) natural system on $\Delta/X_{\bullet}$, geometric realisation always produces a sheaf on $|X_{\bullet}|$ and defines a left exact functor from the category $\thomnat$ of (covariant) natural systems to the category 
$\Shv(|X_{\bullet}|)$ of sheaves on the topological space $|X_{\bullet}|$ (see \cite[Prop. 3.1]{Fi}).
\end{example}

\begin{example}[Parshin-Beilinson adeles of schemes]
Let $X$ be a noetherian scheme and $\Qcoh(X)$ denote the abelian category of quasi-coherent $\mathcal{O}_X$-modules. 
Furthermore let $P(X)$ be the set of points of the scheme $X$. Let $S_{\bullet}(X)$ be the associated simplicial set of flags of irreducible closed subschemes of $X$, ordered by inclusion, given as follows: consider the set of points $P(X)$ of $X$ with the partial order $\geq$ on $P(X)$ defined by $\eta\geq \nu$ if $\nu\in \overline{\{\eta\}}$. Then $S_{\bullet}(X)$ is the simplicial nerve of the partially ordered set $(P(X), \geq)$, with the set of $n$-simplices
$$S(X)_n=\{(\nu_0, \nu_1, \ldots, \nu_n) | \nu_i\in P(X);\, \nu_i\geq \nu_{i+1}\}$$
and the usual face and degeneracy maps $d_i$ and $s_i$ for $0\leq i \leq n$ induced from the partially ordered set structure.
If $X$ is an affine scheme, the flags of $S(X)_{\bullet}$ are just sequences of prime ideals ordered by inclusions. 
Beilinson \cite{Be} (see also \cite{Hu} for more details) constructed for any $K\subset S(X)_{\bullet}$ and any quasi-coherent sheaf
$\mathcal{F}$ on $X$ a space of adeles $\A(K, \mathcal{F})$ which is an abelian group functorial in $\mathcal{F}$. 
Then the groups of local adeles $\A(\{\sigma\}, \mathcal{F})$, for any simplex $\sigma\in S(X)_{\bullet}$,
give rise to a Gabriel-Zisman (covariant) natural system by setting (compare \cite{Hu,Fi})
$$F\colon \Delta/{S(X)_{\bullet}}\rightarrow \Ab, \,\,\, F(\sigma)=\A(\{\sigma\}, \mathcal{F}),$$
which actually is a sheaf and we have that (see \cite[Prop. 2.1.4]{Hu})
$$\A(K, \mathcal{F})\subset \prod_{\sigma\in K}\A(\{\sigma\}, \mathcal{F}).$$
In particular, we can consider the abelian group of $n$-dimensional adeles of $X$ with coefficients in $\mathcal{F}$ defined as
$$\A^n(X, \mathcal{F})= \A(S(X)_n, \mathcal{F}).$$ 
It turns out that the sequence of groups of global adeles $\A^n(X, \mathcal{F})$ on $X$ gives a cosimplicial abelian group $\A^{\bullet}(X, \mathcal{F})$ and therefore a cochain complex.  Its cohomology, which is the Gabriel-Zisman cohomology $H^*_{GZ}(S(X)_{\bullet}, F)$ for $S_{\bullet}(X)$ calculates sheaf cohomology i.e., if $\mathcal{F}$ is a quasi-coherent $\mathcal{O}_X$-module, then we have an isomorphism \cite[Thm 4.2.3]{Hu}
$$H^*(\A^{\bullet}(X, \mathcal{F}))\cong H^*(X, \mathcal{F}).$$
Parshin \cite{Pa} gave first a definition of adeles for smooth proper algebraic surfaces over a perfect field, which was later extended by Beilinson \cite{Be} for arbitrary noetherian schemes.
\end{example}

\begin{example} [Buildings of reductive algebraic groups]
Let $G$ be a reductive algebraic group over the finite field $\F_q$ and $\Rep(G)$ be the category of finite dimensional representations of the finite group of $\F_q$-rational points $G(\F_q)$. Associated to $G$ is a simplicial set $\Delta(G)_{\bullet}$, the combinatorial building of $G$ consisting of inclusion chains in the poset of subgroups of G given by parabolic
subgroups. For any simplex $\sigma\in \Delta(G)_{\bullet}$ we have a parabolic subgroup $P_{\sigma}\subset G$.  Let $R_u(P)$ be the unipotent radical of a parabolic subgroup $P$ and $R_u(P)(\F_q)$ its group of $\F_q$-rational points. Let $M\in \Rep(G)$, then we obtain a (covariant) Gabriel-Zisman natural system by setting
$$F\colon \Delta/{\Delta(G)_{\bullet}}\rightarrow \Rep(G), \,\,\, F(\sigma)= M^{R_u(P_{\sigma})}$$
and inclusion maps for different simplices. Here $M^{R_u(P_{\sigma})}\subset M$ and $F$ turns out to be again a sheaf (see \cite{Fi}) and Gabriel-Zisman cohomology $H^*_{GZ}(\Delta(X)_{\bullet}, F)$ gives the cohomology of the building with coefficients being representations as sheaves on the building (compare \cite{Br,SS}).

\end{example}

\begin{remark}[Higher categories]
  As mentioned above, Gabriel-Zisman cohomology extends and unifies many notions of cohomology of categories. Recall that the factorisation category (also known as the twisted arrow category)
      $\mathfrak F\mathit{act}(\Cc)$ of a category $\Cc$ has objects the morphisms  $f\colon x\to y$ of $\Cc$
and arrows
$(h,k)\colon f\to f'$, where $f'=kfh$ in $\Cc$. Then
  Baues-Wirsching cohomology $H^{*}(\Cc,D)$ was defined in~\cite{BW}, for natural systems of coefficients $D\colon \mathfrak F\mathit{act}(\Cc)\to\Ab$. The relation to Thomason cohomology arises from the existence of a functor
$$
\eta_{\mathscr C}\colon \Delta/\mathscr C\to
\mathfrak F\mathit{act}(\mathscr C)
  $$
  from the category of simplices to the factorisation category of $\mathscr C$, see~\cite{GNT2}.
  We remark that analogous notions will provide extensions of
  Thomason and Baues-Wirsching cohomologies to:
  \begin{itemize}
    \item \emph{2-categories.} One can define a category of simplices $\Delta/\mathscr D$ of a 2-category $\mathscr D$, with objects $a$ given by the lax functors $a\colon [m]\to \mathscr D$ and arrows $a\to b$ given by morphisms $\sigma\colon [m]\to [n]$ of $\Delta$, where $a=b\sigma$. One can also define a factorisation category  
      $\mathfrak F\mathit{act}(\mathscr D)$, with objects the 1-morphisms  $f\colon x\to y$ of $\mathscr D$
and arrows
$(h,k,\xi)\colon f\to f'$,
where  $\xi\colon kfh\to f'$ is a 2-morphism of $\mathscr D$.
      Furthermore we can give a natural transformation
    \[
\eta_{\mathscr D}\colon \Delta/\mathscr D\to
\mathfrak F\mathit{act}(\mathscr D).
    \]
      We can define notions of Thomason and Baues-Wirsching cohomologies for 2-categories $\mathscr D$, with coefficient systems on
      $\Delta/\mathscr D$ and on
      $\mathfrak F\mathit{act}(\mathscr D)$ respectively.
      
    \item \emph{2-Segal spaces}, also known as \emph{decomposition spaces}~\cite{DK,Moe1}. The 2-Segal condition specifies a particular class of simplicial sets more general than nerves of ordinary categories, which are characterised by the 1-Segal condition.
      It was shown recently, in \cite{boors-proc-ams}, that a simplicial set $X$ is 2-Segal if and only if its \emph{edgewise subdivision} is 1-Segal, and we denote the category defined by this edgewise subdivision by $\mathfrak F\mathit{act}(X)$. If $X$ is 1-Segal this agrees with the definition of the category of factorisations above.
      We can define notions of Thomason and Baues-Wirsching cohomologies for 2-Segal spaces $X$,  with coefficient systems on the categories
      $\Delta/X$ and
$\mathfrak F\mathit{act}(X)$ respectively, related once more by a natural transformation
    \[
\eta_{X}\colon \Delta/X\to
\mathfrak F\mathit{act}(X).
    \]
  \end{itemize}
There is also an obvious notion of cohomology of $\infty$-categories: if we model an $\infty$-category by a quasi-category, that is, by an inner-Kan simplicial set, then we can take its Gabriel-Zisman cohomology. We do not see an analogue of Baues-Wirsching cohomology for $\infty$-categories. 
\end{remark}

\section{Spectral sequences for Gabriel-Zisman (co)homology}

\subsection{(Co)homology spectral sequences for maps of simplicial sets}
In this subsection we will derive Leray type Gabriel-Zisman (co)homology spectral sequences for a given map of simplicial sets. In order to do so, we will work first in a more suitable general categorical setting.

Let $\Cc$ and $\Dd$ be small categories, $\Aa$ a complete abelian category and $T\colon  \Cc\rightarrow \Aa$ be a functor. Now let us assume that we also have a functor $u\colon  \Cc\rightarrow \Dd$ together with functors $P\colon  \Cc\rightarrow \Delta$ and $Q\colon  \Dd\rightarrow \Delta$ such that
$P=Q\circ u$ i.e.,\ we have a commutative diagram of the form
$$
\xymatrix{
{\Cc}\ar[rr]^{u}\ar[rd]_P && {\Dd}\ar[ld]^Q   \\ 
&\Delta &}
$$
inducing a commutative diagram between functor categories, where the respective functors are given by precomposition and right Kan extensions
$$
\xymatrix{
*+++{\Fun(\Cc,\Aa)}\ar @<-5pt>[rrrr]_{Ran_u} \ar @<5pt>[rrdd]^-{\hspace*{0.4cm}Ran_P}
&&&& *+++{\Fun(\Dd,\Aa)} \ar @<-5pt>[llll]_{u^*}   \ar @<5pt>[lldd]^-{Ran_Q} \\ \\
&&\Fun(\Delta,\Aa)\ar @<5pt>[rruu]^-{Q^*} \ar @<5pt>[lluu]^-{P^*}
}
$$

It follows immediately from Definition~\ref{defCoRan} and the above that we have an isomorphism
$$H^*(P, T)\cong H^*(Q, Ran_u(T)).$$

From the previous diagram we get now the following Grothendieck composite functor spectral sequence~\cite{G} (compare also~\cite{An,GNT,GNT2}).

\begin{theorem}\label{GrCFSS}
Let $\Cc$ and $\Dd$ be small categories and $T\colon  \Cc\rightarrow \Aa$ be a functor to a complete abelian category. Let $u\colon  \Cc\rightarrow \Dd$ be a functor together with functors $P\colon  \Cc\rightarrow \Delta$ and $Q\colon  \Dd\rightarrow \Delta$ such that $P=Q\circ u$. Then there is a spectral sequence:
$$E_2^{p, q}\cong H^p(Q, Ran^q_u(T))\Rightarrow H^{p+q}(P, T),$$
which is natural in $u$ and $T$ and where $Ran^q_u(T)$ denotes the $q$-th right satellite of $Ran_u(T)$.
\end{theorem}

Dually, using analogue constructions as just described, we obtain also a homology version of the above spectral sequence

\begin{theorem}\label{GrHFSS}
Let $\Cc$ and $\Dd$ be small categories and  $T\colon  \Cc^{op}\rightarrow \Aa$ be a functor to a cocomplete abelian category. Let $u\colon  \Cc\rightarrow \Dd$ be a functor together with functors $P\colon  \Cc\rightarrow \Delta$ and $Q\colon  \Dd\rightarrow \Delta$ such that $P=Q\circ u$. Then there is a spectral sequence:
$$E^2_{p, q}\cong H_p(Q, Lan_q^u(T))\Rightarrow H_{p+q}(P, T),$$
which is natural in $u$ and $T$ and where $Lan_q^u(T)$ denotes the $q$-th left satellite of $Lan^u(T)$.
\end{theorem} 

We will now derive general Leray type spectral sequences for Gabriel-Zisman (co)homology for any map of simplicial sets using the machinery developed above. In special cases, we can in addition also simplify them by using concrete fiber data.
Let us first introduce the following general constructions:

\begin{definition}
Given a map of simplicial sets $f\colon X\to Y$, the {\em fiber functor}
$$
F_{(-)}\colon \Delta/Y\to\Delta^\op\Set
$$
is defined as follows: 

For each object $y\colon \Delta[n]\to Y$
of the simplex category $\Delta/Y$, let $F_y$ be the  fiber of $f$ over $y$, which is the simplicial set
$$
F_y\;\;=\;\;\Delta[n]\times_Y X
\;\;=\;\;\{(\sigma,x)\in\Delta[n]\times X:y\circ\sigma=f(x)\},
$$
given by the pullback
\begin{equation}\label{pbs}
\vcenter{\xymatrix @R-1ex@C+2em{F_y\drpullback\ar[r]^{\bar y}\ar[d]&X\ar[d]^f\\
\Delta[n]\ar[r]_-y&Y.}
}\end{equation}

For each morphism from $y\colon \Delta[n]\to Y$ to $y'\colon \Delta[n']\to Y$, given by $\theta\colon \Delta[n]\to\Delta[n']$ and satisfying $y=y'\circ\theta$, let $F_\theta$ be the simplicial map given as:
$$
\theta\times_Y X\,\colon\;F_y\to F_{y'},\quad (\sigma,x)\mapsto (\theta\circ\sigma,x).
$$
\end{definition}
\begin{remark}

Given a simplicial map $f\colon X\to Y$ and a Gabriel-Zisman natural system $T\colon \Delta /X\to\Aa$ we have induced (covariant) natural systems $T_y$ on the fibers $F_y$, for each object $y\in \Delta/Y$, defined by
$$
T_y=T\circ\Delta/\bar y\colon \Delta/F_y\to \Delta/X\to \Aa.
$$
For each $q\geq0$ we get functors
$$
\h^q_{GZ}(F_{(-)},T_{(-)})
\colon (\Delta/Y)^\op\to \Aa
$$
defined on objects by
$$
y\mapsto H^q_{GZ}(F_{y},T_{y})
$$ 
and on morphisms $\theta$ from $y$ to $y'$ by
$$
\theta^*\colon H^q_{GZ}(F_{y'},T_{y'})\to H^q_{GZ}(F_{y},T_{y}).
$$
since $T_y=\theta^*T_{y'}$.

Dually, given a Gabriel-Zisman natural system  $T\colon (\Delta /X)^\op\to\Aa$ we have induced (contravariant) natural systems $T_y$, and for each $q\geq 0$ get functors
$$
\h_q^{GZ}(F_{(-)},T_{(-)})
\colon \Delta/Y\to \Aa
$$
defined on objects by
$$
y\mapsto H_q^{GZ}(F_{y},T_{y})
$$ 
and on morphisms $\theta$ from $y$ to $y'$ by
$$
\theta_*\colon H_q^{GZ}(F_{y},T_{y})\to H_q^{GZ}(F_{y'},T_{y'}).
$$
\end{remark}

We now make the following definition:

\begin{definition} Let $f\colon  X\rightarrow Y$ be a map of simplicial sets and $T\colon  \Delta/Y\rightarrow \Aa$ a (covariant) Gabriel-Zisman natural system. The map $f$ is called \emph{locally cohomologically constant} if for each morphism $\theta\colon  y\rightarrow y'$ of the simplex category $\Delta/Y$ the induced map in cohomology
$$\theta^*\colon H^q_{GZ}(F_{y'},T_{y'})\stackrel{\cong}\to H^q_{GZ}(F_{y},T_{y})$$
is an isomorphism.
\end{definition}

Let $f\colon  X\rightarrow Y$ be a map of simplicial sets and $T\colon  \Delta/Y\rightarrow \Aa$ a (covariant) Gabriel-Zisman natural system. From the pullback square \eqref{pbs} and functoriality of Gabriel-Zisman cohomology we get an induced map in cohomology
$$H_{GZ}^*(\Delta[n], T_{\Delta[n]})\rightarrow H^*_{GZ}(F_y, (f^*T)_y),$$
where for a given simplex $y\colon \Delta[n]\rightarrow Y$ of $Y$ we let $T_{\Delta[n]}\colon  \Delta /\Delta[n]\rightarrow \Delta /Y \rightarrow \Aa$ be the restricted Gabriel-Zisman natural system and
$(f^*T)_y=(f^*T)\circ \Delta/\bar y\colon  \Delta/F_y\to \Delta/X\to \Aa$ the induced Gabriel-Zisman natural system. We make the following definition:

\begin{definition}
Let $f\colon  X\rightarrow Y$ be a map of simplicial sets and $T\colon  \Delta/Y\rightarrow \Aa$ a (covariant) Gabriel-Zisman natural system.
The map $f$ is called \emph{locally cohomologically trivial} if for every simplex $y\colon \Delta[n]\rightarrow Y$ of $Y$ the induced map in cohomology
$$H_{GZ}^*(\Delta[n], T_{\Delta[n]})\stackrel{\cong}\rightarrow H^*_{GZ}(F_y, (f^*T)_y)$$
is an isomorphism.
\end{definition}

The following useful lemma gives an alternative description of the fiber of a general map of simplicial sets. 

\begin{lemma}\label{GZfib}
Let $f\colon  X\rightarrow Y$ be a map of simplicial sets. The simplex category of a fiber $F_y$ is naturally isomorphic to the left fiber of $\Delta/f\colon \Delta/X\to\Delta/Y$ over the object $y$,
$$
\Delta/F_y\;\cong\;(\Delta/f)/y.
$$
\end{lemma}
\begin{proof}
An object of the left fiber of $\Delta/f$ over $y$ is just a map $\sigma\colon (\Delta/f)(x)\to y$ in the comma category $\Delta/Y$, for some $x\in\Delta/X$, as in the following diagram:
$$
\xymatrix @R-2.2ex@C+5ex{\Delta[m]\ar[r]^-{x}\ar[rdd]|{(\Delta/f)(x)}\ar[dd]_{\sigma}&X\ar[dd]^f\\ \\\Delta[n]\ar[r]_-y&Y.}
$$
Such a diagram may alternatively be interpreted as a map $(\sigma,x)
\colon \Delta[m]\to \Delta[n]\times_Y X$, and hence as an object of the category $\Delta/F_y$. 

Now a morphism in $(\Delta/f)/y$ is just a map $\theta\colon \Delta[m]\to\Delta[m']$ which fits into a diagram of the form
$$
\xymatrix @R-2.2ex@C+5ex{\Delta[m]\ar[dr]|\theta\ar @/^2ex/[drr]^(0.76)x\ar @/_1ex/[dddr]_(0.6)\sigma\\&\Delta[m']\ar[r]_-{x'}\ar[dd]_{\sigma'}&X\ar[dd]^f\\ \\&\Delta[n]\ar[r]_-y&Y.}
$$
This may be interpreted as a morphism $\theta\colon (\sigma,x)\to(\sigma',x')$ in $\Delta/F_y$.
\end{proof}

Now given any map $f\colon  X\rightarrow Y$ of simplicial sets we can derive a general cohomology spectral sequence, which compares the Gabriel-Zisman cohomology of $X$ and $Y$.

\begin{theorem}\label{H^*GZ} 
Let $X$ and $Y$ be simplicial sets and $f\colon  X\rightarrow Y$ be a map of simplicial sets. Let $\Aa$ be a complete abelian category with exact products. Given a Gabriel-Zisman natural system $T\colon  \Delta/X\rightarrow \Aa$ on $X$, there is a cohomology spectral sequence
$$E_2^{p,q}\cong H_{GZ}^p(Y, (R^q (\Delta/f)_*)(T))\Rightarrow H^{p+q}_{GZ}(X, T)$$
which is natural in $f$ and $T$ and where $R^q (\Delta/f)_*=Ran^q_{\Delta/f}$ is the $q$-th right satellite of the right Kan extension $Ran_{\Delta/f}$ along the induced functor  $\Delta/f\colon \Delta/X\to\Delta/Y$ between the simplex categories.
\end{theorem}

\begin{proof}
Let $X$ and $Y$ be simplicial sets, $f\colon  X\rightarrow Y$ be a map of simplicial sets and $\Aa$ be a complete abelian category with exact products. 
With the categories $\Cc=\Delta/X$ and $\Dd=\Delta/Y$ and the functors $P=P_X\colon  \Delta/X\rightarrow \Delta$, $Q=Q_Y\colon  \Delta/Y\rightarrow \Delta$
and $u=\Delta/f\colon  \Delta/X\rightarrow \Delta/Y$ we are exactly in the situation of Theorem~\ref{GrCFSS}, with $P=Q\circ u$ and we get the following commutative diagram:
$$
\xymatrix{
*+++{\Fun(\Delta/X,\Aa)}\ar @<-5pt>[rrrr]_{Ran_{\Delta/f}} \ar @<5pt>[rrdd]^-{Ran_{P_X}}
&&&& *+++{\Fun(\Delta/Y,\Aa)} \ar @<-5pt>[llll]_{(\Delta/f)^*}   \ar @<5pt>[lldd]^-{Ran_{Q_Y}} \\ \\
&&{\Fun(\Delta, \Aa)}\ar @<5pt>[rruu]^-{Q_Y^*} \ar @<5pt>[lluu]^-{P_X^*}
}
$$

Therefore, Theorem~\ref{GrCFSS} gives a spectral sequence of the form
$$E_2^{p, q}\cong H^p(Q_Y, Ran^q_u(T))\Rightarrow H^{p+q}(P_X, T),$$
which is natural in $u$ and $T$.

Identifying the above cohomologies of the functors $Q_Y$ and $P_X$ as Gabriel-Zisman cohomology following Theorem~\ref{GZExt} we get the desired spectral sequence of the form
$$E_2^{p,q}\cong H_{GZ}^p(Y, (R^q (\Delta/f)_*)(T))\Rightarrow H^{p+q}_{GZ}(X, T).$$
and the naturality of the spectral sequence with respect to $f$ and $T$ follows directly from the above identifications.
\end{proof}

\begin{remark} 
If we start with a Gabriel-Zisman natural system which is actually a sheaf (see Remark \ref{sheafrem}), then the above spectral  sequence corresponds to the Leray spectral sequence for sheaf cohomology, in fact, if applying the geometric realisation functor as in Example \ref{ExSheaves} we will obtain the classical Leray spectral sequence for sheaf cohomology of a continuous map of topological spaces.
\end{remark} 

We can identify the $E_2$-term of the spectral sequence further by relating the satellites of the right Kan extension to derived limit data of the fiber of the simplicial map $f$.

\begin{corollary}\label{CorGZcoh}
Let $X$ and $Y$ be simplicial sets and $f\colon  X\rightarrow Y$ be a map of simplicial sets. Let $\Aa$ be a complete abelian category with exact products.  Let $T\colon  \Delta/X\rightarrow \Aa$ be a Gabriel-Zisman natural system on $X$.
Then there is a cohomology spectral sequence of the form
$$E_2^{p,q}\cong H_{GZ}^p(Y,  \h^q_{GZ}(-/(\Delta/f), T\circ Q^{(-)}))\Rightarrow H^{p+q}_{GZ}(X, T)$$
which is natural in $f$ and $T$ and where  
$$\h^q_{GZ}(-/(\Delta/f), T\circ Q^{(-)})= \llim^q_{-/(\Delta/f)}(T\circ Q^{(-)})\colon  \Delta/Y\rightarrow \Aa.$$
\end{corollary}

\begin{proof}
For each simplex $y$ of $\Delta/Y$, let $Q^{(y)}\colon y/(\Delta/f)\to \Delta/X$ be the forgetful functor
and denote by $\h^q_{GZ}(y/(\Delta/f), T\circ Q^{(y)})$ the derived limit
$${\lim}^q\left( y/(\Delta/f)\stackrel{Q^{(y)}}\longrightarrow 
\Delta/X\stackrel{T}\longrightarrow \Aa\right).$$
Using~\cite[Corollary 1.3]{GNT} allows us to
identify the terms in the $E_2$-page of 
the spectral sequence in Theorem~\ref{H^*GZ} as
$$
E_2^{p,q}\cong H_{GZ}^p(Y,  \h^q_{GZ}(-/(\Delta/f), T\circ Q^{(-)})).
$$
and in addition gives us the desired abutment. 
\end{proof}

As a direct consequence, we also have the following general statement for locally cohomologically trivial maps of simplicial sets. This can be seen as a cohomological analogue of Quillen's Theorem A  (see also~\cite{C,Q}) for Gabriel-Zisman cohomology

\begin{proposition} 
Let $f\colon  X\rightarrow Y$ be a map of simplicial sets and $T\colon  \Delta/Y\rightarrow \Aa$ a (covariant) Gabriel-Zisman natural system. If $f$ is locally cohomologically trivial, then $f$ induces an isomorphism in cohomology:
$$H^*_{GZ}(Y, T)\stackrel{\cong}\rightarrow H^*_{GZ}(X, f^*T).$$
\end{proposition}

\begin{proof} For every simplex $y\colon \Delta[n]\to Y$ of the simplex category $\Delta/Y$ we have the following commutative diagram
$$
\xymatrix @R-1ex@C+2em{F_y\drpullback\ar[r]^{\bar y}\ar[d]&X\ar[d]^f\ar[r]^f&Y\ar[d]^{id}\\
\Delta[n]\ar[r]_-y&Y\ar[r]_{id}&Y.}
$$

The naturality of the spectral sequence of Theorem~\ref{H^*GZ} gives a morphism of spectral sequences $
E^{*, *}_r(id, T)\rightarrow E^{*, *}_r(f, T) 
$.
Because $f$ is locally cohomological trivial, we get an isomorphism of $E_2$-pages i.e., 
$
E^{*, *}_2(id, T)
\stackrel{\cong}\rightarrow 
E^{*, *}_2(f, T)
$. Therefore we also get an isomorphism of the abutments, which implies the statement.
\end{proof}

Dually, we can derive a homology spectral sequence computing the Gabriel-Zisman homologies for a simplicial map $f\colon X\rightarrow Y$, which gives the dual version of Theorem~\ref{H^*GZ}.

\begin{theorem}\label{H_*GZ} 
Let $X$ and $Y$ be simplicial sets and $f\colon  X\rightarrow Y$ be a map of simplicial sets. Let $\Aa$ be a cocomplete abelian category with exact coproducts. Given a contravariant Gabriel-Zisman natural system $T\colon  (\Delta/X)^{op}\rightarrow \Aa$ on $X$, there is a homology spectral sequence
$$E^2_{p,q}\cong H^{GZ}_p(Y, (L_q ((\Delta/f)^{op})_*)(T))\Rightarrow H_{p+q}^{GZ}(X, T)$$
which is natural in $f$ and $T$ and where $L_q ((\Delta/f)^{op})_*=Lan_q^{(\Delta/f)^{op}}$ is the $q$-th left satellite of $Lan^{(\Delta/f)^{op}}$, the left Kan extension along the induced functor  $(\Delta/f)^{op}\colon (\Delta/X)^{op}\to(\Delta/Y)^{op}$ between the simplex categories.
\end{theorem}

\begin{proof}Let $X$ and $Y$ be simplicial sets and $\Aa$ be a cocomplete abelian category with exact coproducts. 
Given a map $f\colon  X\rightarrow Y$ of simplicial sets we have the following commutative diagram:
$$
\xymatrix{
*+++{\Fun((\Delta/X)^{op},\Aa)}\ar @<-5pt>[rrrr]_{Lan^{(\Delta/f)^{op}}} \ar @<5pt>[rrdd]^-{\hspace*{0.4cm}\clim^{(\Delta/X)^{op}}}
&&&& *+++{\Fun((\Delta/Y)^{op},\Aa)} \ar @<-5pt>[llll]_{((\Delta/f)^{op})^*}   \ar @<5pt>[lldd]^-{\clim^{(\Delta/Y)^{op}}} \\ \\
&&\Aa\ar @<5pt>[rruu]^-c \ar @<5pt>[lluu]^-c
}
$$
Here, $c$ denotes the respective constant diagram functors and ${((\Delta/f)^{op})}^*$ is pre-composition with $(\Delta/f)^{op}$, the induced functor between the simplex categories of $X$ and $Y$. The other functors in the diagram are the left adjoints of these, given by the limits $\clim^{(\Delta/X)^{op}}$, $\clim^{(\Delta/Y)^{op}}$ and by $Lan^{{(\Delta/f)^{op}}}$, which is the left Kan extension along the functor $(\Delta/f)^{op}$.

We obtain a Grothendieck spectral sequence~\cite{G} for the derived functors of the composite functor
$$\clim^{(\Delta/X)^{op}}(-)=\clim^{(\Delta/Y)^{op}} Lan^{(\Delta/f)^{op}}(-)$$
which can be interpreted as an Andr\'e spectral sequence as constructed in generality in~\cite[Section 1.1]{GNT} (see also~\cite{An} and~\cite{C}).

In our situation here it converges to the homology of the simplex category $\Delta/X$ of the simplicial set $X$ with coefficients being a contravariant Gabriel-Zisman natural system $T$ of
$\Fun((\Delta/X)^{op},\Aa)$. Therefore,~\cite[Theorem 1.4]{GNT} gives a cohomology spectral sequence of the form:
$$
E^2_{p,q}\cong H_p((\Delta/Y)^{op}, (L_q ((\Delta/f)^{op})_*)(T))\Rightarrow H_{p+q}((\Delta/X)^{op}, T)
$$
where $L_q ((\Delta/f)^{op})_*$ is the $q$-th left satellite of $Lan^{(\Delta/f)^{op}}$.

Identifying the homologies of the involved simplex categories $(\Delta/X)^{op}$ and $(\Delta/Y)^{op}$ with the Gabriel-Zisman homologies of the given simplicial sets $X$ and $Y$ using Proposition~\ref{GZTor} finally gives us the homology spectral sequence
$$E^2_{p,q}\cong H^{GZ}_p(Y, (L_q ((\Delta/f)^{op})_*)(T))\Rightarrow H_{p+q}^{GZ}(X, T).$$
The naturality of the spectral sequence with respect to $f$ and $T$ follows directly from the construction.
\end{proof}

Again, we can identify the $E^2$-term of the spectral sequence by relating the satellites of the left Kan extension to derived colimit data of the fiber of the simplicial map $f$.

\begin{corollary}
Let $X$ and $Y$ be simplicial sets and $f\colon  X\rightarrow Y$ be a map of simplicial sets. Let $\Aa$ be a cocomplete abelian category with exact coproducts.  Let $T\colon  (\Delta/X)^{op}\rightarrow \Aa$ be a contravariant Gabriel-Zisman natural system on $X$.
Then there is a homology spectral sequence of the form
$$E^2_{p,q}\cong H^{GZ}_p(Y,  \h_q^{GZ}((\Delta/f)/-, T\circ Q_{(-)}))\Rightarrow H_{p+q}^{GZ}(X, T)$$
which is natural in $f$ and $T$ and where  
$$\h_q^{GZ}((\Delta/f)/-, T\circ Q_{(-)})= \clim_q^{(\Delta/f)^{op}/-}(T\circ Q_{(-)}^{op})\colon  (\Delta/Y)^{op}\rightarrow \Aa.$$
\end{corollary}

\begin{proof}
For each simplex $y$ of $(\Delta/Y)^{op}$, let $Q_{(y)}^{op}\colon (\Delta/f)^{op}/y\to (\Delta/X)^{op}$ be the forgetful functor
and denote by $\h_q^{GZ}((\Delta/f)/y), T\circ Q_{(y)})$ the derived colimit
$${\clim}_q\left( (\Delta/f)^{op}/y\stackrel{Q_{(y)}^{op}}\longrightarrow 
(\Delta/X)^{op}\stackrel{T}\longrightarrow \Aa\right).$$
Using~\cite[Corollary 1.5]{GNT} allows us now to
identify the terms in the $E^2$-page of 
the above spectral sequence as 
$$
E^2_{p,q}\cong H^{GZ}_p(Y,  \h_q^{GZ}((\Delta/f)/-, T\circ Q_{(-)})).
$$
while the spectral sequence converges to the same abutment. 
\end{proof}

\subsection{Specialisation of coefficient systems and spectral sequences}
In this final subsection we will specialise the general coefficient systems in order to identify the $E_2$-terms of the (co)homology spectral sequence further.  The classical Leray-Serre spectral sequences for Kan fibrations of simplicial sets will appear as a special case. Let us start by introducing some useful special Gabriel-Zisman natural systems in order to simplify our Leray type spectral sequences in various situations.

\begin{definition} Let $X$ be a simplicial set and $\Mm$ be a category.   A (covariant) Gabriel-Zisman natural system $T\colon  \Delta/X\rightarrow \Mm$ on $X$ is called \emph{invertible} or a \emph{(covariant) local system} if it sends all morphism of $\Delta/X$ to isomorphisms of $\Mm$.

Dually,  a (contravariant) Gabriel-Zisman natural system $T\colon  (\Delta/X)^{op}\rightarrow \Mm$ on $X$ is called \emph{invertible}
or a \emph{(contravariant) local system} if it sends all morphism of $(\Delta/X)^{op}$ to isomorphisms of $\Mm$.
\end{definition}

Let $X$ be a simplicial set and $\Mm$ be a category.  Let $T\colon  \Delta/X\rightarrow \Mm$ be a (covariant) invertible Gabriel-Zisman natural system on $X$. Then we can define the functor $T^{-1}\colon  (\Delta/X)^{op}\rightarrow \Mm$, whose value on objects is the same as for the functor $T$ and whose value on a morphism $\alpha$ of $\Delta/X$ is $T^{-1}(\alpha)=T(\alpha)^{-1}$.
Dually, given a (contravariant) invertible Gabriel-Zisman natural system $T\colon  (\Delta/X)^{op}\rightarrow \Mm$ on $X$, we can define similarly the functor $T^{-1}\colon  \Delta/X \rightarrow \Mm$, whose value on objects is the same as for the functor $T$ and whose value on a morphism $\alpha$ of $(\Delta/X)^{op}$ is $T^{-1}(\alpha)=T(\alpha)^{-1}$.

The following proposition gives an alternative description of Gabriel-Zisman (co)homology for invertible coefficient functors (compare also~\cite[App. II. 4.4]{GZ}).

\begin{proposition}\label{GZcohoid}
Let $X$ be a simplicial set and $\Aa$ be a complete and cocomplete abelian category with exact products and coproducts. 

Given a (covariant) local system $T\colon  \Delta/X\rightarrow \Aa$ on $X$, there is an isomorphism
$$H_{ GZ }^* (X, T)=H^*(\Delta/X, T)\cong H^*((\Delta/X)^{op}, T^{-1}),$$
natural in $X$ and $T$.

Dually, given a (contravariant) local system $T\colon  (\Delta/X)^{op}\rightarrow \Aa$ on $X$, there is an isomorphism $$H_*^{GZ}(X, T)=H_*((\Delta/X)^{op} , T)\cong H_*(\Delta/X, T^{-1}),$$
natural in $X$ and $T$.
\end{proposition}

\begin{proof}
In the case of homology, this follows verbatim as in the proof of the proposition in~\cite[App. II.4.4]{GZ} by interpreting Gabriel-Zisman homology of simplicial sets as Thomason homology of small categories applied to the respective categories of simplicies (see~\cite{GNT2}). The case for cohomology follows analogous from the dual arguments using $\Delta/X$ instead of $(\Delta/X)^{op}$.
\end{proof}

Now let $f\colon  X\rightarrow Y$ be a map of simplicial sets, which is locally cohomologically constant and let $T\colon  \Delta/X\rightarrow \Aa$ be a Gabriel-Zisman natural system.  Then we obtain an induced covariant functor for each $q\geq0$
$$
\h^q_{GZ}(F_{(-)},T_{(-)})^{-1}
\colon \Delta/Y\to \Aa
$$
defined on objects by
$$
y\mapsto H^q_{GZ}(F_{y},T_{y})
$$ 
and which maps morphisms $\theta$ from $y$ to $y'$ in $\Delta/X$ to the induced inverse morphism 
$$
(\theta^*)^{-1}\colon H^q_{GZ}(F_{y},T_{y})\to H^q_{GZ}(F_{y'},T_{y'}).
$$

This allows us to derive the following cohomology spectral sequence for locally cohomologically constant maps of simplicial sets 

\begin{proposition}
Let $X$ and $Y$ be simplicial sets and $f\colon  X\rightarrow Y$ be a map of simplicial sets, which is locally cohomologically constant. Let $\Aa$ be a complete abelian category with exact products and $T\colon  \Delta/X\rightarrow \Aa$ be a Gabriel-Zisman natural system on $X$.
Then there is a cohomology spectral sequence of the form
$$E_2^{p,q}\cong H_{GZ}^p(Y,  \h^q_{GZ}(F_{(-)},T_{(-)})^{-1})\Rightarrow H^{p+q}_{GZ}(X, T)$$
which is natural in $f$ and $T$. 
\end{proposition}

\begin{proof}
This follows from the identification of the $E_2$-page of the general Leray type spectral sequence in Corollary~\ref{CorGZcoh} for the particular case of a given locally cohomologically constant map of simplicial sets using the natural isomorphism
$$\h^q_{GZ}(-/(\Delta/f), T\circ Q^{(-)}) \cong  \h^q_{GZ}(F_{(-)},T_{(-)})^{-1}).$$ 
The abutment of the spectral sequence does not change and it is again natural in $f$ and $T$.
\end{proof}

Finally, we will derive the Leray-Serre spectral sequences of a Kan fibration of simplicial sets in cohomology and homology with local coefficients from our general setting (compare also~\cite[App. II.4.4]{GZ},~\cite{Dr}).

Let $f\colon  X\rightarrow Y$ first be any map of simplicial sets. Furthermore, let $\Aa$ be a complete abelian category with exact products and $T\colon  \Delta/X\rightarrow \Aa$ a covariant local system on $X$.
Then Lemma~\ref{GZfib} and Proposition~\ref{GZcohoid} imply
$$(R^n(\Delta/f)_*)(T^{-1})(y)=H^n(F_y, T|_{F_y}),$$
where $T|_{F_y}$ is given as the composition
$$\Delta/F_y \stackrel{\Delta/pr_2}\longrightarrow \Delta/X\stackrel{T}\longrightarrow \Aa$$
If in addition $f\colon  X\rightarrow Y$ is a Kan fibration of simplicial sets, then 
$$(R^n(\Delta/f)_*)(T^{-1})\colon  y\mapsto H^n(F_y, T|_{F_y})$$ 
induces a covariant local system $\h^q_{GZ}(f, T)\colon  \Delta/X\rightarrow \Aa$. 
Dually, we can make similar considerations starting with a cocomplete abelian category with exact coproducts and a contravariant local system $T\colon  (\Delta/X)^{op}\rightarrow \Aa$ on $X$. We then obtain a contravariant local system $\h_q^{GZ}(f, T)\colon  (\Delta/X)^{op}\rightarrow \Aa$ induced by $$(L_n(\Delta/f)^*)(T^{-1})\colon  y\mapsto H_n(F_y, T|_{F_y}),$$
where $T|_{F_y}$ is now given as the composition
$$(\Delta/F_y)^{op}\stackrel{(\Delta/pr_2)^{op}}\longrightarrow (\Delta/X)^{op}\stackrel{T}\longrightarrow \Aa.$$
The following follows now from Theorem~\ref{H^*GZ} and Theorem~\ref{H_*GZ} and recovers the Leray-Serre spectral sequence of a Kan fibration (compare also \cite{Dr}).

\begin{proposition}[Leray-Serre spectral sequence]
Let $f\colon  X\rightarrow Y$ be a map of simplicial sets, which is a Kan fibration. Let $\Aa$ be a complete abelian category with exact products and $T\colon  \Delta/X\rightarrow \Aa$ be a covariant local system on $X$. 
Then there is a cohomology spectral sequence of the form
$$E_2^{p,q}\cong H_{GZ}^p(Y,  \h^q_{GZ}(f, T))\Rightarrow H^{p+q}_{GZ}(X, T)$$
which is natural in $f$ and $T$. 

Dually, let $\Aa$ be a cocomplete abelian category with exact coproducts and $T\colon  (\Delta/X)^{op}\rightarrow \Aa$ be a contravariant local system on $X$. 
Then there is a homology spectral sequence of the form
$$E^2_{p,q}\cong H^{GZ}_p(Y,  \h_q^{GZ}(f, T))\Rightarrow H_{p+q}^{GZ}(X, T)$$
which is natural in $f$ and $T$. 
\end{proposition}

Let us finally remark that when taking geometric realisation again we will recover the Leray-Serre spectral sequence for fibrations of topological spaces.

\end{document}